\documentclass[11pt,leqno]{amsart}%
\usepackage{amsmath}
\usepackage{amsfonts}
\usepackage{amssymb}
\usepackage{graphicx}%
\providecommand{\U}[1]{\protect\rule{.1in}{.1in}}

\newtheorem{theorem}{Theorem}

\newtheorem{corollary}[theorem]{Corollary}

\newtheorem{remark}[theorem]{Remark}

\begin{document}

\title[Rigidity theorem for the sharp Sobolev constant]{An alternative proof of a rigidity theorem for the sharp Sobolev constant}
\author{Stefano Pigola}
\address{Dipartimento di Fisica e Matematica\\
Universit\`a dell'Insubria - Como\\
via Valleggio 11\\
I-22100 Como, Italy.}
\email{stefano.pigola@uninsubria.it}

\author{Giona Veronelli}
\address{Dipartimento di Matematica\\
Universit\`a degli Studi di Milano\\
via Saldini 50\\
I-20133 Milano, Italy.}
\email{giona.veronelli@unimi.it}
\date{\today}

\begin{abstract}
We provide a somewhat geometric proof of a rigidity theorem by M. Ledoux and
C. Xia concerning complete manifolds with non-negative Ricci curvature
supporting an Euclidean-type Sobolev inequality with (almost) best Sobolev constant.
Using the same technique we also generalize Ledoux-Xia result to complete manifolds with
asymptotically non-negative curvature.
\end{abstract}

\maketitle

\section{Introduction}
A \ Riemannian manifold $\left(  M,\left\langle ,\right\rangle \right)  $ of
dimension $\dim M=m>p\geq1$ is said to support an Euclidean-type Sobolev
inequality if there exists a constant $C_{M}>0$ such that, for every $u\in
C_{c}^{\infty}\left(  M\right)  $,%
\begin{equation}
\left(  \int_{M}\left\vert u\right\vert ^{p^{\ast}}d\mathrm{vol}\right)
^{\frac{1}{p^{\ast}}}\leq C_{M}\left(  \int_{M}\left\vert \nabla u\right\vert
^{p}d\mathrm{vol}\right)  ^{\frac{1}{p}},\label{eq_sob}%
\end{equation}
where%
\begin{equation}\nonumber
p^{\ast}=\frac{mp}{m-p}\label{p-star}%
\end{equation}
and $d\mathrm{vol}$ denotes the Riemannian meausure of $M$. Clearly,
(\ref{eq_sob}) implies that there exists a continuous imbedding $W^{1,p}%
\left(  M\right)  \hookrightarrow L^{p^{\ast}}\left(  M\right)  $, and can be
expressed in the equivalent form%
\begin{equation}\nonumber
C_{M}^{-p}\leq\inf_{u\in\Lambda}\int_{M}\left\vert \nabla u\right\vert
^{p}d\mathrm{vol},\label{3'}%
\end{equation}
where%
\begin{equation}\nonumber
\Lambda=\left\{  u\in L^{p^{\ast}}\left(  M\right)  :\left\vert \nabla
u\right\vert \in L^{p}\text{ and}\int_{M}\left\vert u\right\vert ^{p^{\ast}%
}d\mathrm{vol}_{M}=1\right\}  .\label{4}%
\end{equation}

The validity of (\ref{eq_sob}), as well as the best value of the Sobolev
constant $C_{M}$, have intriguing and deep connections with the geometry
of the underlying manifold, many of which are discussed in the
excellent lecture notes \cite{H-NA}. See also \cite{Le-Toulouse}
for a survey in the more abstract perspective of Markov diffusion processes, and \cite{GT-JGA}
for the relevance of (\ref{eq_sob}) in the $L^{p,q}$-cohomology theory. For instance, we note that a complete manifold with non-negative Ricci curvature (but, in fact, a certain amount of negative curvature is allowed) and supporting an Euclidean-type Sobolev inequality is necessarily connected at infinity. This fact can be proved using (non-linear) potential theoretic arguments; see \cite{PRS-Progress}, \cite{PST-topology}.

 It
is  known (see e.g. Proposition 4.2 in \cite{H-NA}) that%
\begin{equation}
C_{M}\geq K(m,p),\label{constants}%
\end{equation}
where $K\left(  m,p\right)  $ is the best constant in the corresponding
Sobolev inequality of $\mathbb{R}^{m}$. It was discovered by M. Ledoux,
\cite{Le-CAG}, that for complete manifolds of non-negative Ricci curvature,
the equality in (\ref{constants}) forces $M$ to be isometric to $\mathbb{R}%
^{m}$. This important rigidity result has been generalized by C. Xia,
\cite{X-Illinois}, by showing that, in case $C_{M}$ is sufficiently close to
$K\left(  m,p\right)  $, then $M$ is diffeomorphic to $\mathbb{R}^{m}$. The first aim of this
note is to provide a simple and somewhat geometric proof of the Ledoux-Xia
rigidity result.\smallskip

\noindent\textbf{Notation. } In what follows, having fixed a reference origin
$o\in M$, we set $r\left(  x\right)  =\rm{dist}_{M}\left(  x,o\right)  $ and we
denote by $B_{t}$ and $\partial B_{t}$ the geodesic ball and sphere of radius
$t>0$ centered at $o$. The corresponding balls and spheres in the $m$%
-dimensional Euclidean space are denoted by $\mathbb{B}_{t}$ and
$\partial\mathbb{B}_{t}$. Finally, the symbols $V\left(  B_{t}\right)  $ and
$A\left(  \partial B_{t}\right)  $ stand, respectively, for the Riemannian
volume of $B_{t}$ and the $\left(  m-1\right)  $-dimensional Hausdorff measure
of $\partial B_{t}$.\medskip

\begin{theorem}
\label{th_1}Let $\left(  M,\left\langle ,\right\rangle \right)  $ be a
complete, $m$-dimensional Riemannian manifold, $m>p>1$. Assume that
$^{M}\operatorname{Ric}\geq0$ and that the Euclidean-type Sobolev inequality
(\ref{eq_sob}) holds on $M$. Then
\begin{equation}
V(\mathbb{B}_{t})\geq V(B_{t})\geq\left(  \frac{K(m,p)}{C_{M}}\right)
^{m}V(\mathbb{B}_{t}). \label{vol_comp}%
\end{equation}
In particular, if $C_{M}$ is sufficiently close to $K\left(  m,p\right)  $
then $M$ is diffeomorphic to $\mathbb{R}^{m}$ and, in case $C_{M}=K(m,p),$ $M$
is isometric to $\mathbb{R}^{m}$.
\end{theorem}
Actually, using the same technique, we shall prove that a lower control on the volume
of geodesic balls from a fixed origin can be obtained even if we allow a certain
amount of negative curvature. More precisely, we will prove the following

\begin{theorem}
\label{th_10}Let $\left(  M,\left\langle ,\right\rangle \right)  $ be a
complete, $m$-dimensional Riemannian manifold, $m\geq 3$, with
\begin{equation}\label{Ric}
{}^M\operatorname{Ric}(y)\geq -(m- 1)G(r(y))\ \textrm{on }M
\end{equation}
for some non-negative function $G\in C^0([0,+\infty))$. Assume  that $G$ satisfies the integrability condition
\begin{equation}\label{int_G}
\int_0^{\infty}tG(t)dt=b<+\infty
\end{equation}
and that the Euclidean-type Sobolev inequality
(\ref{eq_sob}) holds on $M$, for some $1<p<m$. Then
\begin{equation}
e^{mb}V(\mathbb{B}_{t})\geq V(B_{t})\geq \hat C(m,p,C_M,b) V(\mathbb{B}_{t}), \label{vol_comp'}%
\end{equation}
where
\[\hat C(m,p,C_M,b)\to \left(\frac{K(m,p)}{C_M}\right)^m \text{ as }b\to 0.\]
\end{theorem}

Combining Theorem \ref{th_10} with Theorem 3.1 in \cite{Zhu}, see also \cite{Be}, we immediately deduce the next rigidity result.

\begin{corollary}\label{coro_10}
Given $m\geq 3$, $m>p$, there exist constants $b_0=b_0(m,p)>0$ and
$\varepsilon_0=\varepsilon_0(m,p)>0$ such that, if $M$ is an
$m$-dimensional complete manifold supporting the Sobolev
inequality (\ref{eq_sob}) with $C_M\leq K(m,p)+\varepsilon$ and
such that
\[
{}^M\operatorname{Sect}\geq -G(r)\ \textrm{on }M,
\]
where $G$ satisfies \eqref{int_G} for some $b\leq b_0$, then $M$
is diffeomorphic to $\mathbb{R}^{m}$.
\end{corollary}

\section{Proof of the Ledoux-Xia theorem}
Recall  that, in $\mathbb{R}^{m}$, the equality in (\ref{eq_sob}) with the
best constant $C_{\mathbb{R}^{m}}=K(m,p)$, is realized by the (radial)
Bliss-Aubin-Talenti functions $\phi_{\lambda}\left(  x\right)  =\varphi_{\lambda
}(|x|)$ for every $\lambda>0$, where $|x|$ is the Euclidean norm of $x$ and
$\varphi_{\lambda}(t)$ are the real-valued functions defined as
\begin{equation}\nonumber
\varphi_{\lambda}\left(  t\right)  =\frac{\beta\left(  m,p\right)
\lambda^{\frac{m-p}{p^{2}}}}{\left(  \lambda+t^{\frac{p}{p-1}}\right)
^{\frac{m}{p}-1}}.\label{6}%
\end{equation}
If we choose $\beta(m,p)>0$ such that%
\begin{equation}\label{norm_one}
\int_{\mathbb{R}^{m}}\phi_{\lambda}^{p^{\ast}}(x)dx=1
\end{equation}
then%
\begin{equation}\nonumber
K\left(  m,p\right)  ^{-p}=\int_{\mathbb{R}^{m}}\left\vert \varphi_{\lambda
}^{\prime}\left(  \left\vert x\right\vert \right)  \right\vert ^{p}%
dx\label{5}%
\end{equation}
and, by the standard calculus of variations, the extremal functions
$\phi_{\lambda}$ obey the (nonlinear) Yamabe-type equation%
\begin{equation}
^{\mathbb{R}^{m}}\!\Delta_{p}\phi_{\lambda}=-K\left(  m,p\right)  ^{-p}%
\phi_{\lambda}^{p^{\ast}-1},\label{yamabe}%
\end{equation}
where%
\[
\Delta_{p}u=\operatorname{div}\left(  \left\vert \nabla u\right\vert
^{p-2}\nabla u\right)
\]
stands for the $p$-Laplacian of a given function $u$.

Define $\hat{\phi}_{\lambda}:M\rightarrow\mathbb{R}$ as $\hat{\phi}_{\lambda
}(x):=\varphi_{\lambda}(r(x))$. The idea of our proof is simply to apply Karp version of
Stokes theorem, \cite{K}, to the vector field $X_{\lambda}:=\hat{\phi}_{\lambda}\left\vert \nabla\phi_{\lambda}\right\vert
^{p-2}\nabla\hat{\phi}_{\lambda}$, once we have observed that,
by \eqref{yamabe} and the Laplacian comparison theorem, each function $\hat{\phi}_{\lambda}$ on $M$ satisfies%
\[
\Delta_{p}\hat{\phi}_{\lambda}\geq-K\left(  m,p\right)  ^{-p}\hat{\phi}_{\lambda}%
^{p^{\ast}-1}.
\]
This leads directly to inequality (2.2) in \cite{X-Illinois} and the argument can be completed essentially as explained by Xia.
\begin{proof} [Proof of Theorem \ref{th_1}]
We claim that
\begin{equation}
\left\{
\begin{array}
[c]{ll}%
\text{(i) } & \int_{M}\hat{\phi}_{\lambda}^{p^{\ast}}d\mathrm{vol}\leq1\text{, }\\
\text{(ii) } & |\nabla\hat{\phi}_{\lambda}|\in L^{p}\left(  M\right)  \text{,
}.
\end{array}
\right.  \label{integ}%
\end{equation}
Indeed, since $^{M}Ric\geq0$, according to the
Bishop-Gromov comparison theorem, \cite{Ch}, \cite{PRS-Progress}, $A\left(  \partial B_{t}\right)  /A\left(
\partial\mathbb{B}_{t}\right)  $ is a decreasing function of $t>0$ and,
therefore,%
\begin{equation}
A\left(  \partial B_{t}\right)  \leq A\left(  \partial\mathbb{B}_{t}\right)
\text{, }V\left(  B_{t}\right)  \leq V\left(  \mathbb{B}_{t}\right)
.\label{volumes}%
\end{equation}
The validity of (\ref{integ}) follows from the co-area formula. Furthermore, since ${}^M\operatorname{Ric}\geq0$, by Laplacian comparison, \cite{PRS-Progress}, $\Delta r \leq (m-1)/r$ pointwise on $M\setminus \rm{cut}(o)$ and weakly on all of $M$. This means that
\begin{equation}
-\int_M\left\langle \nabla r,\nabla\eta\right\rangle d\mathrm{vol}\leq\int_M\eta\frac{m-1}{r} d\mathrm{vol}\label{laplace}%
\end{equation}
for all $0\leq\eta\in W^{1,2}_{c}(M)$. Let $0\leq\xi\in
C^{\infty}_c(M)$ to be chosen later and apply \eqref{laplace} with
\begin{equation}\label{eta}
\eta=-\left(\xi\hat{\phi}_{\lambda}\right)|\varphi'_{\lambda}(r)|^{p-2}\varphi'_{\lambda}(r),
\end{equation}
thus obtaining
\begin{align}\label{eq_xi}
&\int_M\varphi'_{\lambda}(r)|\varphi'_{\lambda}(r(y))|^{p-2}\left\langle \nabla r,\nabla\left(\xi\hat{\phi}_{\lambda}\right)\right\rangle d\mathrm{vol}\\
&\leq -\int_M
|\varphi'_{\lambda}(r(y))|^{p-2}\left[(p-1)\varphi''_{\lambda}(r)+\frac{(m-1)}{r}\varphi'_{\lambda}(r)\right]\left(\xi\hat{\phi}_{\lambda}\right)d\mathrm{vol}.\nonumber
\end{align}
On the other hand, according to \eqref{yamabe},
\begin{equation}\label{rad_yamabe}
|\varphi'_{\lambda}(t)|^{p-2}\left\{(p-1)\varphi''_{\lambda}(t)+\frac{m-1}{t}\varphi'_{\lambda}(t)\right\}=-K(m,p)^{-p}\varphi_{\lambda}^{p^{\ast}-1}(t)
\end{equation}
for all $t>0$, and inserting into (\ref{eq_xi}) gives
\begin{align}\label{eq_div}
&-\int_M K(m,p)^{-p}\xi\hat{\phi}_{\lambda}^{p^{\ast}}d\mathrm{vol}\\
%&=\int_M |\varphi'_{\lambda}(r(y))|^{p-2}\left[\varphi''_{\lambda}(r)+\frac{(m-1)}{r}\varphi'_{\lambda}(r)\right]\left[\xi\hat{\phi}_{\lambda}\right]d\mathrm{vol}\nonumber\\
&\leq -\int_M |\varphi'_{\lambda}(r(y))|^{p-2}\varphi'_{\lambda}(r)\left\langle \nabla r,\hat{\phi}_{\lambda}\nabla\xi+\xi\varphi'_{\lambda}(r)\nabla r \right\rangle d\mathrm{vol}\nonumber\\
&\leq
-\int_M\xi\left|\nabla\hat{\phi}_{\lambda}\right|^pd\mathrm{vol}-\int_M
\hat{\phi}_{\lambda}\varphi'_{\lambda}(r)|\varphi'_{\lambda}(r(y))|^{p-2}\left\langle
\nabla r,\nabla\xi\right\rangle d\mathrm{vol}.\nonumber
\end{align}
Now choose $\xi=\xi_R$ such that $\xi\equiv 1$ on $B_R$,
$\xi\equiv 0$ on $M\setminus B_{2R}$ and $|\nabla\xi|<2/R$. By
Cauchy-Schwarz inequality and volume comparison we get
\begin{align*}
&\left|\int_M \hat{\phi}_{\lambda}\varphi'_{\lambda}(r)|\varphi'_{\lambda}(r(y))|^{p-2}\left\langle
\nabla r,\nabla\xi\right\rangle d\mathrm{vol}\right|\\
%&\leq \frac{2}{R}\int_{B_{2R}\setminus B_{R}} \hat{\phi}_{\lambda}\varphi'_{\lambda}(r)|\varphi'_{\lambda}(r(y))|^{p-2}d\mathrm{vol}\\
&\leq
\frac{2}{R}\int_R^{2R}\varphi_{\lambda}(t)(-\varphi'_{\lambda}(t))^{p-1}A(\partial
\mathbb B_t)dt\to 0\nonumber
\end{align*}
as $R\to\infty$, for every $m>p$. Then, taking the limits as
$R\to\infty$ in \eqref{eq_div} we obtain
\begin{align*}\label{eq_div}
K(m,p)^{-p}\int_M \hat{\phi}_{\lambda}^{p^{\ast}}d\mathrm{vol}\geq
\int_M\left|\nabla\hat{\phi}_{\lambda}\right|^pd\mathrm{vol},
\end{align*}
proving that%
\begin{equation}
\frac{\int_{M}\left\vert \nabla\hat{\phi}_{\lambda}\right\vert
^{p}}{\int _{M}\hat{\phi}_{\lambda}^{p^{\ast}}}\leq K\left(
m,p\right)
^{-p}.\label{12}%
\end{equation}
On the other hand, because of (\ref{integ}), we can use $\hat
{\phi}_{\lambda}$ into (\ref{eq_sob}) and get%
\begin{equation}
\frac{\int_{M}\left\vert \nabla\hat{\phi}_{\lambda}\right\vert ^{p}}{\int
_{M}\hat{\phi}_{\lambda}^{p^{\ast}}}\geq\frac{\int_{M}\left\vert \nabla
\hat{\phi}_{\lambda}\right\vert ^{p}}{\left(  \int_{M}\hat{\phi}_{\lambda
}^{p^{\ast}}\right)  ^{\frac{p}{p^{\ast}}}}\geq C_{M}^{-p}.\label{13}%
\end{equation}
Combining (\ref{12}) and (\ref{13}) we obtain
\begin{equation}\nonumber
1\geq\int_{M}\hat{\phi}_{\lambda}^{p^{\ast}}\geq\left(  \frac{K(m,p)}{C_{M}%
}\right)  ^{m}.\label{estimate}%
\end{equation}
From this latter, using (\ref{norm_one}), the co-area formula and integrating by parts, it follows that
\begin{align}
0 &  \leq\int_{M}\left(  \frac{C_{M}}{K(m,p)}\right)  ^{m}\hat{\phi}_{\lambda
}^{p^{\ast}}d\mathrm{vol}-\int_{\mathbb{R}^{m}}\phi_{\lambda}^{p^{\ast}%
}(x)dx\label{eq_0leq}\\
&  =\int_{0}^{\infty}v_{M}(t)V(\mathbb{B}_{t})\frac{d}{dt}\left(
-\varphi_{\lambda}^{p^{\ast}}(t)\right)  dt\nonumber
\end{align}
where, by Bishop-Gromov, the function%
\[
v_{M}(t):=\left[  \left(  \frac{C_{M}}{K(m,p)}\right)  ^{m}\frac{V(B_{t}%
)}{V(\mathbb{B}_{t})}-1\right]
\]
is non-increasing. In order to prove (\ref{vol_comp}), it is enough to show that
\[\lim_{t\rightarrow\infty
}v_{M}(t)\geq0.\]
By contradiction, suppose there exist positive constants
$\epsilon$ and $T$ such that $v_{M}(t)\leq-\epsilon$ for all $t\geq T$. Define $T_0=\sup\{t<T:v_{M}(t)\geq 0\}$. Then $0<T_0<T$ and
\begin{align}
&  \int_{0}^{\infty}v_{M}(t)V(\mathbb{B}_{t})\frac{d}{dt}\left(
-\varphi_{\lambda}^{p^{\ast}}(t)\right)  dt\label{eq_0T}\\
&  \leq v_{M}(0)\int_{0}^{T_0}V(\mathbb{B}_{t})\frac{d}{dt}\left(
-\varphi_{\lambda}^{p^{\ast}}(t)\right)  dt\nonumber\\
&  -\epsilon\int_{T}^{\infty}V(\mathbb{B}_{t})\frac{d}{dt}\left(
-\varphi_{\lambda}^{p^{\ast}}(t)\right)  dt.\nonumber
\end{align}
Observe that the $1$-parameter family of functions%
\[
V(\mathbb{B}_{t})\frac{d}{dt}\left(  -\varphi_{\lambda}^{p^{\ast}}(t)\right)
=\omega_{m}\frac{mp}{p-1}\beta\left(  m,p\right)  ^{p^{\ast}}\lambda^{\frac
{m}{p}}\frac{t^{\frac{1}{p-1}+m}}{\left(  \lambda+t^{\frac{p}{p-1}}\right)
^{m+1}}%
\]
is decreasing in $\lambda$, provided $\lambda>>1$. Then we can apply the dominated
convergence theorem to deduce
\begin{equation}
\lim_{\lambda\rightarrow+\infty}\int_{0}^{T}V(\mathbb{B}_{t})\frac{d}%
{dt}\left(  -\varphi_{\lambda}^{p^{\ast}}(t)\right)  dt=0.\label{eq_limit1}%
\end{equation}
On the other hand, using the co-area formula once again,%
\[
\int_{0}^{\infty}V(\mathbb{B}_{t})\frac{d}{dt}\left(  -\varphi_{\lambda
}^{p^{\ast}}(t)\right)  dt=\int_{0}^{\infty}A(\partial\mathbb{B}_{t}%
)\varphi_{\lambda}^{p^{\ast}}(t)dt=\int_{\mathbb{R}^{m}}\phi_{\lambda
}^{p^{\ast}}\left(  x\right)  dx=1,
\]
for all $\lambda>0$. Therefore, by (\ref{eq_limit1}),%
\begin{equation}
\lim_{\lambda\rightarrow+\infty}\int_{T}^{\infty}V(\mathbb{B}_{t})\frac{d}%
{dt}\left(  -\varphi_{\lambda}^{p^{\ast}}(t)\right)  dt=1.\label{eq_limit2}%
\end{equation}
Using (\ref{eq_limit1}) and (\ref{eq_limit2}) into (\ref{eq_0T}) we conclude
that, up to choosing $\lambda>0$ large enough,%
\[
\int_{0}^{\infty}v_{M}(t)V(\mathbb{B}_{t})\frac{d}{dt}\left(  -\varphi
_{\lambda}^{p^{\ast}}(t)\right)  dt<0,
\]
which contradicts (\ref{eq_0leq}). We have thus proven the validity of
(\ref{vol_comp}).

Now, if $C_{M}$ is sufficiently close to $K\left(  m,p\right)  $ an
application of a result by J. Cheeger and T. Colding, \cite{CC-JDG}, yields that $M$ is
diffeomorphic to $\mathbb{R}^{m}$. On the other hand, if $C_{M}=K(m,p)$,
(\ref{vol_comp}) gives $\operatorname{Vol}(\mathbb{B}_{t})=\operatorname{Vol}%
(B_{t})$ for all $t>0$, and we conclude that $M$ is isometric to
$\mathbb{R}^{m}$ by the equality case in the Bishop-Gromov comparison theorem.
\end{proof}
\begin{remark}
\rm{
It should be noted that, in order to prove the isometry with $\mathbb{R}^{m}$, we can merely use the
existence of a single minimizing function $\varphi_{\lambda}$. Indeed, assume $C_M=K(m,p)$. Since $v_{M}(t)\geq0$ and the selected function $-\varphi_{\lambda}^{p^{\ast}}(t)$ is increasing, from (\ref{eq_0leq}) we immediately deduce that $v_{M}(t)$ vanishes identically. The presence of an entire family  $\varphi_{\lambda}$ of minimizers well behaving with respect to the parameter $\lambda$ is actually needed to reach the general volume estimate and the consequent diffeomorphism with $\mathbb{R}^{m}$. On the other hand, as Ledoux pointed out, the original argument is oriented towards the conjecture that a sharp Euclidean Sobolev inequality, without any curvature assumption, implies a sharp Euclidean volume lower bound.
}
\end{remark}

\section{The case of manifolds with asymptotically non-negative curvature}
This section is devoted to a proof of Theorem \ref{th_10}. To this end, we follow exactly the strategy we used to prove Theorem \ref{th_1}.
Clearly, this time we have to take into account the (small) perturbations of (\ref{yamabe}) introduced by the negative curvature.

\begin{proof} [Proof of Theorem \ref{th_1}]
Let $h\in C^2([0,+\infty))$ be the solution of the
problem
\[
\begin{cases}
&h''(t)-G(t)h(t)=0,\\
&h(0) = 0,\  h'(0) = 1,
\end{cases}
\]
and consider the $m$-dimensional model manifold $M_h$ defined as
$M_h:=\left(\mathbb R\times\mathbb
S^{m-1},ds^2+h^2(s)d\theta^2\right)$, where $d\theta^2$ is the
standard metric on $\mathbb S^{m-1}$. We shall use an index '$h$' to
denote objects and quantities referred to $M_h$. Thus, we denote
by $\mathbb{B}^h_{t}$ and $\partial\mathbb{B}^h_{t}$ the geodesic
ball and sphere of radius $t>0$ in $M_h$. Moreover we introduce
the family of functions
$\phi_{\lambda,h}:M_h\rightarrow\mathbb{R}$ defined by
$\phi_{\lambda,h }((s,\theta)):=\varphi_{\lambda}(s)$.
For later purposes, we recall that, \cite {Zhu}, \cite{PRS-Progress},
\begin{equation}\label{star'}
\mathrm{V}(  \mathbb{B}_{t}^{h})  \geq\mathrm{V}\left(
\mathbb{B}_{t}\right)  \text{, }t\geq0.
\end{equation}
Furthermore, we observe that, according to the Bishop-Gromov comparison theorem and its generalizations,
\cite{Ch}, \cite{PRS-Progress}, $A\left(\partial B_{t}\right)
/A\left(\partial \mathbb{B}^h_{t}\right)  $ is a decreasing
function of $t>0$ and the following relations hold
\begin{align}\label{volumes'}%
&A\left(  \partial B_{t}\right)  \leq A(  \partial\mathbb{B}^h_{t}) \leq e^{b(m-1)}A\left(  \partial\mathbb{B}_{t}\right),\\
&V\left(  B_{t}\right)  \leq V(  \mathbb{B}^h_{t})\leq e^{bm}V\left(  \mathbb{B}_{t}\right)\nonumber.
\end{align}
By the co-area formula, these imply
\begin{equation}
\left\{
\begin{array}
[c]{ll}%
\text{(i) } & \int_{M}\hat{\phi}_{\lambda}^{p^{\ast}}d\mathrm{vol}\leq \int_{M_h}\phi_{\lambda,h}^{p^{\ast}}d\mathrm{vol}_h \leq e^{b(m-1)}\text{, }\\
\text{(ii) } & |\nabla\hat{\phi}_{\lambda}|\in L^{p}\left(  M\right),\\
\text{(iii) } & r(x)^{-1}\widehat{\phi}_{\lambda}|\nabla\widehat{\phi}_{\lambda}%
|^{p-1}\in L^1(M)\\
\text{(iv)} & \int_{B_{R}}\widehat{\phi}_{\lambda}|\nabla\widehat{\phi}_{\lambda}%
|^{p-1}d\mathrm{vol}=o\left(  R\right)  \text{, as } R\rightarrow+\infty.

\end{array}
\right.  \label{integ'}%
\end{equation}
Here $d\mathrm{vol}_h$ stands for the Riemannian measure on $M_h$.
We need also to recall from \cite{Car} that the validity of (\ref{eq_sob})
implies that there exists a (small)
constant $\gamma=\gamma(m,p,C_M)>0$ (depending continuously on $C_M$) such that
\begin{equation}\label{lower_vol'}
V(B_t)\geq \gamma V(\mathbb B_t).
\end{equation}

Now, by Laplacian comparison, assumption \eqref{Ric} yields
\[\Delta r \leq \frac{(m-1)e^b}{r}
\]
pointwise on $M\setminus \rm{cut}(o)$ and weakly on all of $M$. This means that
\begin{equation}
-\int_M\left\langle \nabla r,\nabla\eta\right\rangle d\mathrm{vol}\leq\int_M\eta\frac{(m-1)e^b}{r} d\mathrm{vol},\label{laplace'}%
\end{equation}
for all $0\leq\eta\in W^{1,2}_{c}(M)$. Let $0\leq\xi\in
C^{\infty}_c(M)$ to be chosen later and apply \eqref{laplace'} with $\eta$ defined in (\ref{eta})
thus obtaining
\begin{align}\label{eq_xi'}
&\int_M\varphi'_{\lambda}(r)|\varphi'_{\lambda}(r)|^{p-2}\left\langle \nabla r,\nabla\left(\xi\hat{\phi}_{\lambda}\right)\right\rangle d\mathrm{vol}\\
&\leq -\int_M
|\varphi'_{\lambda}(r)|^{p-2}\left[(p-1)\varphi''_{\lambda}(r)+\frac{(m-1)e^b}{r}\varphi'_{\lambda}(r)\right]\left(\xi\hat{\phi}_{\lambda}\right)d\mathrm{vol}.\nonumber
\end{align}
Whence, inserting (\ref{rad_yamabe}) gives
\begin{align}\label{eq_div'}
&-\int_M K(m,p)^{-p}\xi\hat{\phi}_{\lambda}^{p^{\ast}}d\mathrm{vol}+\int_M
\hat{\phi}_{\lambda}\varphi'_{\lambda}(r)|\varphi'_{\lambda}(r)|^{p-2}\frac{(m-1)(e^b-1)}{r}\xi d\mathrm{vol}\\
%&=\int_M |\varphi'_{\lambda}(r)|^{p-2}\left[\varphi''_{\lambda}(r)+\frac{(m-1)}{r}\varphi'_{\lambda}(r)\right]\left[\xi\hat{\phi}_{\lambda}\right]d\mathrm{vol}\nonumber\\
%&\leq -\int_M |\varphi'_{\lambda}(r)|^{p-2}\varphi'_{\lambda}(r)\left\langle \nabla %r,\hat{\phi}_{\lambda}\nabla\xi+\xi\varphi'_{\lambda}(r)\nabla r \right\rangle d\mathrm{vol}\nonumber\\
&\leq\nonumber
-\int_M\xi\left|\nabla\hat{\phi}_{\lambda}\right|^pd\mathrm{vol}-\int_M
\hat{\phi}_{\lambda}\varphi'_{\lambda}(r)|\varphi'_{\lambda}(r)|^{p-2}\left\langle
\nabla r,\nabla\xi\right\rangle d\mathrm{vol}\\\nonumber
&\leq -\int_M\xi\left|\nabla\hat{\phi}_{\lambda}\right|^pd\mathrm{vol}-\int_M
\hat{\phi}_{\lambda}\varphi'_{\lambda}(r)|\varphi'_{\lambda}(r)|^{p-2}|\nabla\xi| d\mathrm{vol}.
\nonumber
\end{align}
Now choose $\xi=\xi_R$ such that $\xi\equiv 1$ on $B_R$,
$\xi\equiv 0$ on $M\setminus B_{2R}$ and $|\nabla\xi|<2/R$. Then, taking the limits as
$R\to+\infty$ in \eqref{eq_div'} and recalling (\ref{integ'}) we obtain
\begin{align*}
&\int_M\left|\nabla\hat{\phi}_{\lambda}\right|^pd\mathrm{vol}-K(m,p)^{-p}\int_M \hat{\phi}_{\lambda}^{p^{\ast}}d\mathrm{vol}\\
&\leq
(m-1)\left(\tfrac{m-p}{p-1}\right)^{p-1}\beta^p(e^b-1)\lambda^{\frac{m-p}{p}}\int_M (\lambda+ r^{\frac{p}{p-1}})^{-(m-1)}d\mathrm{vol},
\end{align*}
proving that%
\begin{equation}
\frac{\int_{M}\left\vert \nabla\hat{\phi}_{\lambda}\right\vert
^{p}}{\int _{M}\hat{\phi}_{\lambda}^{p^{\ast}}}\leq K\left(
m,p\right)
^{-p}+ C_1\label{12'},%
\end{equation}
where we have set
\begin{equation}\label{c_1'}
C_1(m,p,\lambda,b):=(m-1)\left(\tfrac{m-p}{p-1}\right)^{p-1}\beta^{-\tfrac{p^2}{m-p}}(e^b-1)\frac{\int_M (\lambda+ r^{\frac{p}{p-1}})^{-(m-1)}d\mathrm{vol}}{\lambda\int_M (\lambda+ r^{\frac{p}{p-1}})^{-m}d\mathrm{vol}}.
\end{equation}
By \eqref{volumes'} and computing explicitly the integrals on
$\mathbb R^m$, we get
\begin{align*}
\int_M (\lambda+ r^{\frac{p}{p-1}})^{-(m-1)}d\mathrm{vol}
&\leq\int_0^{\infty}\frac{A(\partial \mathbb B_t) e^{b(m-1)}}{(\lambda+t^{\frac{p}{p-1}})^{m-1}}\\
&=e^{b(m-1)}A(\partial \mathbb B_1) \lambda^{-\frac{m-p}{p}}\frac{\Gamma(m-\frac{m}{p})\Gamma(\frac{m}{p}-1)}{\frac{p}{p-1}\Gamma(m-1)},
\end{align*}
where $\Gamma$ denotes the Euler Gamma function. On the other hand,
\begin{equation}\label{bord'}
\frac{V(B_t)}{\left(\lambda+t^{\frac{p}{p-1}}\right)^m}\leq \frac{V( \mathbb{B}_t^h)}{\left(\lambda+t^{\frac{p}{p-1}}\right)^m}\leq e^{bm}\frac{V(\mathbb{ B}_t)}{\left(\lambda+t^{\frac{p}{p-1}}\right)^m}\to 0,
\end{equation}
as $t\to\infty$.
Therefore, we can integrate by parts using the co-area formula, apply \eqref{lower_vol'}, integrate by parts again and compute explicitly the integrals on
$\mathbb R^m$, thus obtaining
\begin{align*}
\int_M (\lambda+ r^{\frac{p}{p-1}})^{-m}d\mathrm{vol}
&=\int_0^{\infty} V(B_t)\left(-\frac{d}{dt}(\lambda+ t^{\frac{p}{p-1}})^{-m}\right)dt\\
&\geq \gamma \int_0^{\infty} V(\mathbb B_t)\left(-\frac{d}{dt}(\lambda+ t^{\frac{p}{p-1}})^{-m}\right)dt\\
&=\gamma\int_0^{\infty} A(\partial \mathbb B_t)\frac{d}{dt}(\lambda+ t^{\frac{p}{p-1}})^{-m}dt\\
&=\gamma A(\partial \mathbb B_1) \lambda^{-\frac{m}{p}}\frac{\Gamma(m-\frac{m}{p})\Gamma(\frac{m}{p})}{\frac{p}{p-1}\Gamma(m)}.
\end{align*}
Inserting into (\ref{c_1'}) and (\ref{12'}), it follows that
\begin{equation}
\frac{\int_{M}\left\vert \nabla\hat{\phi}_{\lambda}\right\vert
^{p}}{\int _{M}\hat{\phi}_{\lambda}^{p^{\ast}}}\leq K\left(
m,p\right)
^{-p}+ C_2\label{12a'},%
\end{equation}
where
\[
C_2(m,p,b,C_M)
:= \tfrac{(m-1)^2p}{m-p}\left(\tfrac{m-p}{p-1}\right)^{p-1}\beta^{-\tfrac{p^2}{m-p}}(e^b-1)\frac{e^{b(m-1)}}{\gamma}.
\]
On the other hand, because of (\ref{integ'}), we can use $\hat
{\phi}_{\lambda}$ into (\ref{eq_sob}) and get%
\begin{equation}
\frac{\int_{M}\left\vert \nabla
\hat{\phi}_{\lambda}\right\vert ^{p}}{\left(  \int_{M}\hat{\phi}_{\lambda
}^{p^{\ast}}\right)  ^{\frac{p}{p^{\ast}}}}\geq C_{M}^{-p}.\label{13'}%
\end{equation}
Combining (\ref{12a'}) and (\ref{13'}) we obtain
\begin{equation}\nonumber
\int_{M}\hat{\phi}_{\lambda}^{p^{\ast}}\geq C_3^{-1},\label{estimate'}%
\end{equation}
with
\begin{equation}\nonumber
C_3(m,p,b,C_M):=\left[\left(  \frac{C_{M}%
}{K(m,p)}\right)  ^{p}+C_M^pC_2\right]^{m/p}
\end{equation}
From this latter, using (\ref{integ'}), (\ref{bord'}),  the co-area formula and integrating by parts, it follows that
\begin{align}
0 &  \leq C_3e^{b(m-1)}\int_{M} \hat{\phi}_{\lambda
}^{p^{\ast}}d\mathrm{vol}-\int_{M_h}\phi_{\lambda,h}^{p^{\ast}}d\mathrm{vol}_h\label{eq_0leq'}\\
%&  =\left.C_3e^{b(m-1)}V(B_{t})\varphi_{\lambda}^{p^{\ast}}(t)\right|_0^{\infty}-\left.V(\mathbb B^h_{t})\varphi_{\lambda}^{p^{\ast}}(t)\right|_0^{\infty}\nonumber\\
%&+ \int_{0}^{\infty}v_{M,h}(t)V(\mathbb{B}^h_{t})\frac{d}{dt}\left(
%-\varphi_{\lambda}^{p^{\ast}}(t)\right)  dt\nonumber\\
&=\int_{0}^{\infty}v_{M,h}(t)V(\mathbb{B}^h_{t})\frac{d}{dt}\left(
-\varphi_{\lambda}^{p^{\ast}}(t)\right)  dt\nonumber
\end{align}
where, by Bishop-Gromov, the function%
\[
v_{M,h}(t):=\left[  C_3e^{b(m-1)}\frac{V(B_{t}%
)}{V(\mathbb{B}^h_{t})}-1\right]
\]
is non-increasing. In view of (\ref{star'}), in order to prove (\ref{vol_comp'}), it's enough
to show that $\lim_{t\rightarrow\infty }v_{M,h}(t)\geq0$. By
contradiction, suppose there exist positive constants $\epsilon$
and $T$ such that $v_{M,h}(t)\leq-\epsilon$ for all $t\geq T$. In
this assumption, $T_0:=\sup\{t<T:v_{M,h}(t)\geq0\}$ is well
defined and $0<T_0<T$. Then
\begin{align}
&
\int_{0}^{\infty}v_{M,h}(t)V(\mathbb{B}^h_{t})\frac{d}{dt}\left(
-\varphi_{\lambda}^{p^{\ast}}(t)\right)  dt\label{eq_0T'}\\
&  \leq
v_{M,h}(0)\int_{0}^{T_0}V(\mathbb{B}^h_{t})\frac{d}{dt}\left(
-\varphi_{\lambda}^{p^{\ast}}(t)\right)  dt\nonumber\\
& -\epsilon\int_{T}^{\infty}V(\mathbb{B}^h_{t})\frac{d}{dt}\left(
-\varphi_{\lambda}^{p^{\ast}}(t)\right)  dt.\nonumber
\end{align}
Observe that
\begin{align}
&\lim_{\lambda\rightarrow+\infty}\int_{0}^{T_0}V(\mathbb{B}^h_{t})\frac{d}%
{dt}\left(  -\varphi_{\lambda}^{p^{\ast}}(t)\right)
dt\label{eq_limit1'}
\\&\leq e^{bm}
\lim_{\lambda\rightarrow+\infty}\int_{0}^{T_0}V(\mathbb{B}_{t})\frac{d}%
{dt}\left(  -\varphi_{\lambda}^{p^{\ast}}(t)\right)
dt=0.\nonumber%
\end{align}
On the other hand, using \eqref{star'} and the co-area
formula, we have
\begin{align*}
\int_{0}^{\infty}V(\mathbb{B}^h_{t})\frac{d}{dt}\left(
-\varphi_{\lambda
}^{p^{\ast}}(t)\right)  dt
&\geq\int_{0}^{\infty}V(\mathbb{B}_{t})\frac{d}{dt}\left(
-\varphi_{\lambda
}^{p^{\ast}}(t)\right)  dt\\
&\geq\int_{0}^{\infty}A(\partial\mathbb{B}_{t}%
)\varphi_{\lambda}^{p^{\ast}}(t)dt\\
&=\int_{\mathbb{R}^{m}}\phi_{\lambda
}^{p^{\ast}}\left(  x\right)  dx
=1,
\end{align*}
for all $\lambda>0$. Therefore, by (\ref{eq_limit1'}),%
\begin{equation}
\lim_{\lambda\rightarrow+\infty}\int_{T}^{\infty}V(\mathbb{B}^h_{t})\frac{d}%
{dt}\left(  -\varphi_{\lambda}^{p^{\ast}}(t)\right)  dt
\geq
%\lim_{\lambda\rightarrow+\infty}\{\gamma - \int_{0}^{T_0}V(\mathbb{B}_{t})\frac{d}%
%{dt}\left(  -\varphi_{\lambda}^{p^{\ast}}(t)\right)  dt\}
%=
1.\label{eq_limit2'}%
\end{equation}
Inserting (\ref{eq_limit1'}) and (\ref{eq_limit2'}) into (\ref{eq_0T'}) we conclude
that, up to choosing $\lambda>0$ large enough,%
\[
\int_{0}^{\infty}v_{M,h}(t)V(\mathbb{B}^h_{t})\frac{d}{dt}\left(
-\varphi _{\lambda}^{p^{\ast}}(t)\right)  dt<0,
\]
which contradicts (\ref{eq_0leq'}). Setting $\hat
C(m,p,C_M,b):=C_3^{-1}e^{-b(m-1)}$, we have thus proven the
validity of (\ref{vol_comp'}).
\end{proof}

\end{document}